\documentclass[12pt]{amsart}
\usepackage{amssymb,latexsym}
\usepackage{enumerate}

\makeatletter
\@namedef{subjclassname@2010}{%
  \textup{2010} Mathematics Subject Classification}
\makeatother

\newtheorem{thm}{Theorem}[section]
\newtheorem{cor}[thm]{Corollary}

\newtheorem{pro}[thm]{Proposition}

\theoremstyle{definition}
\newtheorem{defin}[thm]{Definition}
\newtheorem{rem}[thm]{Remark}
\newtheorem{exa}[thm]{Example}

\numberwithin{equation}{section}

\begin{document}
\begin{center}
{\bf \Large Convexity properties related \\
\vskip 5mm
to extremal functions}
\vskip 10mm
{by \ {\sc Miros\l aw Baran}$^*$ \ and \ {\sc Leokadia Bialas-Ciez}$^{**}$ \
}

\vskip 9mm
$^*$ Pedagogical  University, Faculty of Mathematics, Physics and Technical Sciences
\\  30-084 Krak\'ow, Podchor\c{a}\.zych 2, POLAND \\ {\it e-mail}: miroslaw.baran.tarnow@gmail.com
\vskip 2mm
$^{**}$Jagiellonian University, Faculty of Mathematics and Computer Science\\
30-348 Krak\'ow, \L ojasiewicza 6, POLAND\\
{\it e-mail}: leokadia.bialas-ciez@uj.edu.pl
\end{center}

\vskip 15mm \noindent {\small \bf Abstract.} {\small In 1962 J\'ozef Siciak introduced in Transactions of the AMS \cite{SiciakTransactions} his famous polynomial extremal function, which was intensively investigated and applied in pluripotential theory and polynomial approximations related to the Chebyshev norm on the compacts in $\mathbb{C}^N$. In particular, starting from middle seventies the Siciak extremal function was one of the most important tool to investigate the behavior of derivatives of polynomials. The pioneer was Wies\l{}aw Ple\'sniak in his researches of quasianalytic functions in the sense of Bernstein. In the circle of papers (most important joint with Wies\l{}aw Paw\l{}ucki) there were shawn deep connections between behavior Siciak extremal function near compact $K$ and bounds for derivatives of polynomials. In particular, in 1990 W. Ple\'sniak \cite{Pl90}\ introduced condition (P) which is equivalent to Markov property of compact $K$. In the same paper there was stated a problem which property of Siciak's extremal function are necessary to Markov's property. In particular, thus Markov sets are non pluripolar that is Siciak's extremal function is finite at every point. Much more  stronger question is on H\"older continuity of the logarithm of the Siciak extremal function, which plays a role of the pluricomplex Green function (see \cite{Klimek} for excelent presentation). This problem can be formulate in more general case of arbitrary norms $q$ on the space of polynomials. In the present paper we, continuing our earlier researches, investigate the connection between behavior of  generalizations of Siciak's function and the behavior of norms of derivatives of polynomials. In particular we get some deep properties of Markov factors $M_n(q,k)$ related to the main problems. One of the main result is the Kolmogorov-Landau type property of $M_n(q,k)^{1/k}$ which is a condition on the triangle sequence of family of derivatives of polynomials not for particular polynomials as for direct analogons of the Kolmogorov-Landau remarkable inequalities: $\log M_n(q,k)^{1/k}\leq \log const. + (1-\frac{\log k}{\log n})\log M_n(q,1)+\frac{\log k}{\log n}M_n(q,n)^{1/n},\ 1\leq k\leq n$.It seems that this condition is satisfied for arbitrary norm $q$. Separately this condition (a weaker version is sufficient) gives nothing. But if we assume that $q$ has A. Markov's property with respect to $q$ and satisfies a condition $C(q)>0$ then $q$ posseses Vladimir Markov property. In the case $q(P)=||P||_E$ this means that non pluripolar Markov sets possese H\"older continuous pluricomplex Green function (in the one dimensional case Markov sets are not polar \cite{Ada}). This is presented in last section. Earlier we investigate a number of extremal functions, between them related to Ple\'sniak condition and to V. Markov's property. We shall consider mainly one dimensional case, but there is no problem to generalize for many variables.}

\vskip 10mm

\noindent {\small {\bf Key words:} Green function, capacity, Chebyshev constant, Markov inequality, Ple\'sniak property.

}

\vskip 13mm

\section{Introduction.} The vector space of polynomials of $N$ variables with complex coefficients we shall denote by $\mathbb{P}(\mathbb{C}^N)$. Then 
$\mathbb{P}_n(\mathbb{C}^N)=\{ P\in\mathbb{P}(\mathbb{C}^N):\ \deg P\leq n\}$. If we consider a norm $q(P)=||P||$ in $\mathbb{P}(\mathbb{C}^N)$ we shall get a normed space $X_q=(\mathbb{P}(\mathbb{C}^N),q)$ and finite dimensional spaces $X_{q,n}=(\mathbb{P}_n(\mathbb{C}^N),q)$ with the dual $X_{q,n}^*$. Thus, as it is well known, $q(P)=\sup\{|\Lambda (P)|:\ \Lambda\in X_{q,n}^*,\ ||\Lambda ||^*=1\}$.

A main motivation of this paper and a lot of earlier researches is to get  bounds of partial derivatives of polynomials in spaces $X_{q,n}$ and to investigate them. We can consider a bound for $|\Lambda (D^\alpha P)|$, where $\Lambda\in X_{q,n}^*$ or a supremum $|\Lambda (D^\alpha P|)$, $\Lambda\in \mathcal{A}\subset X_{q,n}^*$, where $\mathcal{A}$ is a bounded set. In particular we shall consider $||D^\alpha P||$. 

A basic observation is an obvious fact
$$P(z+\zeta )=\sum\limits_{|\alpha|\leq\deg P}\frac{1}{\alpha!}D^\alpha P(z)\zeta^\alpha ,$$ where as usually $\alpha!=\alpha_1!\cdots\alpha_N!,\ \zeta^\alpha=\zeta_1^{\alpha_1}\cdots \zeta_N^{\alpha_N}$.

Next step is a choice of a norm in $\mathbb{C}^N$, consider the unit ball $\mathbb{B}$ with respect to this norm and next we can take a Borel probabilistic measure which is supported on $\partial_s\mathbb{B}$ the Shilov boundary of $\mathbb{B}$.  Actually there is well known that the complex equilibrum measure $\mu_{\mathbb{B}}=(2\pi)^{-N}(dd^cV_{\mathbb{B}})^N$ has this property.

We shall present a few examples.  To do this let us recall some standard notations. The unit disk in $\mathbb{C}$ is $\mathbb{D}$, the unit ciricle is $\mathbb{T}$, $\mathbb{D}_N$ is the polidisk in $\mathbb{C}^N$, while $\mathbb{T}^N$ is the $N$ dimensional tori, which is equal $\mathbb{T}^N=\text{\rm extr}(\mathbb{D}_N)$. By $\mathbb{B}_N$ is denoted the unit Euclidean ball (with respect to the standard inner product), $\mathbb{S}_N=\partial\mathbb{B}_N=\text{\rm extr}(\mathbb{B}_N)$.
\medskip

\begin{exa} $||z||=||z||_\infty=\max(|z_1|,\dots ,|z_N|)$,\ $\mathbb{B}=\mathbb{D}_N,\ \partial_s\mathbb{B}=\partial_s\mathbb{D}_N=\mathbb{T}^N$,
$$\mu_{\partial_s\mathbb{B}}=d\sigma_1\cdots d\sigma_N,$$ where $d\sigma_j$ is the normalized arclength mesure on $\mathbb{T}$, that is 
$$\int_{\mathbb{C}^N}\varphi(z)d\mu_{\partial_s\mathbb{B}}(z)=\left(\frac{1}{2\pi}\right)^N\int_{[0,2\pi ]^N}\varphi (e^{i\theta_1},\dots,e^{i\theta_N})d\theta_1\dots d\theta_N.$$
Now we have (equivalently Cauchy integral formula can be used)
$$\left(\frac{1}{2\pi}\right)^N\int_{[0,2\pi ]^N}P(z+(r_1e^{i\theta_1},\dots r_Ne^{i\theta_N}))e^{-i\theta\cdot\alpha}d\theta =r_1^{\alpha_1}\cdots\ r_N^{\alpha_N}\frac{1}{\alpha!}D^{\alpha }P(z).$$
Hence 
$$D^{\alpha }P(z)=\alpha!r_1^{-\alpha_1}\cdots\ r_N^{-\alpha_N}\left(\frac{1}{2\pi}\right)^N\int_{[0,2\pi ]^N}P(z+(r_1e^{i\theta_1},\dots r_Ne^{i\theta_N}))e^{-i\theta\cdot\alpha}d\theta ,$$
$$|\Lambda (D^{\alpha }P(z))|\leq \alpha!r_1^{-\alpha_1}\cdots\ r_N^{-\alpha_N}\left(\frac{1}{2\pi}\right)^N\int_{[0,2\pi ]^N}|\Lambda (P(z+(r_1e^{i\theta_1},\dots r_Ne^{i\theta_N})))|d\theta ,$$
$$||D^\alpha P(z)||\leq \alpha!r_1^{-\alpha_1}\cdots\ r_N^{-\alpha_N}\left(\frac{1}{2\pi}\right)^N\int_{[0,2\pi ]^N}||P(z+(r_1e^{i\theta_1},\dots r_Ne^{i\theta_N}))||d\theta ,$$
$$||D^\alpha P(z)||\leq \alpha!r_1^{-\alpha_1}\cdots\ r_N^{-\alpha_N}\left(\left(\frac{1}{2\pi}\right)^N\int_{[0,2\pi ]^N}||P(z+(r_1e^{i\theta_1},\dots r_Ne^{i\theta_N}))||^pd\theta \right)^{1/p},\ p\geq 1.$$ In particular

$$||D^\alpha P(z)||\leq  \alpha!r_1^{-\alpha_1}\cdots\ r_N^{-\alpha_N}\max\limits_{\theta\in[0,2\pi ]^n}||P(z+(r_1e^{i\theta_1},\dots r_Ne^{i\theta_N}))||$$
$$\leq \alpha!r_1^{-\alpha_1}\cdots\ r_N^{-\alpha_N}\varphi_n (q,(r_1,\dots ,r_N))||P(z)||,$$ where 
$$ \varphi_n (q,(r_1,\dots ,r_N)):=\sup\{||P(z+\zeta)||:\ |\zeta_1|\leq r_1,\dots,\ |\zeta_N|\leq r_N,\ \deg P\leq n,||P(z)||\leq 1\}$$
$$=\sup\{|\Lambda (P(z+\zeta))|:\ \Lambda\in X_{q,n}^*,||\Lambda||^*=1,\ |\zeta_1|\leq r_1,\dots,\ |\zeta_N|\leq r_N,\ \deg P\leq n,||P(z)||\leq 1\}.$$
Simirally, if $\Lambda\in X_{q,n}^*,||\Lambda||^*=1$ then we put 
$$\varphi_n(q,\Lambda ,(r_1,\dots ,r_n)):=\sup\{|\Lambda (P(z+\zeta))|:\ |\zeta_1|\leq r_1,\dots,\ |\zeta_N|\leq r_N,\ \deg P\leq n,||P(z)||\leq 1\}.$$ Thus
$$\varphi_n (q,(r_1,\dots ,r_N))=\sup\{\varphi_n(q,\Lambda ,(r_1,\dots ,r_n)):\  \Lambda\in X_{q,n}^*,||\Lambda||^*=1\}.$$ Finally, if $1\leq p\leq\infty$ then we define
$\varphi_n(q,p,\Lambda ,(r_1,\dots ,r_n))$ by
$$\sup\left(\left(\frac{1}{2\pi}\right)^N\int_{[0,2\pi ]^N}|\Lambda (P(z+(r_1e^{i\theta_1},\dots,r_Ne^{i\theta_N})|^pd\theta \right)^{1/p}$$ and $\varphi_n(q,p,(r_1,\dots ,r_n))$
to being equal to
$$\sup\limits_{||\Lambda||^*=1}\left(\left(\frac{1}{2\pi}\right)^N\int_{[0,2\pi ]^N}|\Lambda(P(z+(r_1e^{i\theta_1},\dots,r_Ne^{i\theta_N}))|^pd\theta \right)^{1/p}$$
$$\leq \sup\left(\left(\frac{1}{2\pi}\right)^N\int_{[0,2\pi ]^N}||P(z+(r_1e^{i\theta_1},\dots,r_Ne^{i\theta_N})||^pd\theta \right)^{1/p},$$
 where in both cases the supremum is taken over all polynomials $P$ with 
$1\leq \deg P\leq n,||P(z)||\leq 1$. A specially important is the case $p=2$.
\end{exa}

\begin{exa} Now let $||z||=||z||_2=\left(|z_1|^2+\dots +_|z_N|^N\right)^{1/2}$,\ $\mathbb{B}=\mathbb{B}_N,\ \partial_s\mathbb{B}=\mathbb{S}_N=S^{2N-1}$, $B_N=\{x\in\mathbb{R}^N:\ ||x||_2\leq 1\}=\mathbb{B}_N\cap\mathbb{R}^N$, $B_N^+=\{z\in B_N:\ x_j\geq 0,\ j=1,\dots ,N\}$.
$$\mu_{\mathbb{B}_N}=d\sigma,$$ where $d\sigma$ is the normalized surfaces  mesure on $\mathbb{S}_N$ ($|\mathbb{S}_N|=|S^{2N-1}|=2\pi^N/(N-1)!$), that is 
$$\int_{\mathbb{C}^N}\varphi(z)d\mu_{\mathbb{B}_N}(z)=$$
$$\frac{(N-1)!}{2\pi^N}\int_{[0,2\pi ]^N}\int_{S^{N-1}_+}\varphi (\rho_1e^{i\theta_1},\dots,\rho_Ne^{i\theta_N})\rho_1\cdots\rho_Nd\rho d\theta_1\dots d\theta_N.$$
Here $S^{N-1}_+=\{x\in S^{N-1}:\ x_j\geq 0,\ j=1,\dots, N\}$, $d\rho$ is the standard surface measure on $S^{N-1}$. Let us recall that 
$$\int_{S^{N-1}_+}f(\rho_1,\dots,\rho_N)d\rho =\int_{B_N^+}\frac{f(\rho_1,\dots,\rho_{N-1},\sqrt{1-\rho_1^2-\dots-\rho_{N-1}^2})}{\sqrt{1-\rho_1^2-\dots-\rho_{N-1}^2}}d\rho_1\dots d\rho_{N-1}.$$
Now we can write
$$\frac{1}{|S^{2N-1}|}\int_{S^{2N-1}}P(z+r\eta )(\eta_1/|\eta_1|)^{-\alpha_1}\cdots(\eta_N/|\eta_N|)^{-\alpha_N}d\sigma(\eta)=$$
$$\frac{1}{S^{2N-1}}\int_{[0,2\pi]^N}\int_{B_{N-1}^+}P(z+r(\rho_1e^{i\theta_1},\dots,\rho_{N-1}e^{i\theta_{N-1}},(1-\rho_1^2-\dots\rho_{N-1}^2)^{1/2}e^{i\theta_N}))$$
$$\rho_1\cdots\rho_{N-1}e^{-i\theta\cdot\alpha}d\rho_1\cdots\rho_{N-1}d\theta_1\cdots d\theta_N$$
$$=\frac{1}{2}r^{|\alpha |}\frac{1}{\binom{N-1+|\alpha|/2}{N-1}}\frac{\Gamma (\alpha_1/2)\cdots\Gamma(\alpha_N/2)}{\Gamma (|\alpha|/2)}D^\alpha P(z),$$ where
$$\chi_{\alpha}(\rho_1,\dots,\rho_{N-1})=\rho_1^{\alpha_1+1}\cdots\rho_{N-1}^{\alpha_{N-1}+1}(1-\rho_1^2-\dots -\rho_{N-1}^2)^{\alpha_N/2}.$$
Hence

$$D^\alpha P(z)=2r^{-|\alpha|}\binom{N-1+|\alpha|/2}{N-1}\frac{\Gamma(|\alpha|/2)}{\Gamma (\alpha_1/2)\cdots\Gamma(\alpha_N/2)}\cdot $$
$$\frac{1}{|S^{2N-1}|}\int_{S^{2N-1}}P(z+r\eta )(\eta_1/|\eta_1|)^{-\alpha_1}\cdots(\eta_N/|\eta_N|)^{-\alpha_N}d\sigma(\eta)$$ and 
$$\left|\Lambda (D^\alpha P(z))\right|\leq 2r^{-|\alpha|}\binom{N-1+|\alpha|/2}{N-1}\frac{\Gamma(|\alpha|/2)}{\Gamma (\alpha_1/2)\cdots\Gamma(\alpha_N/2)}\frac{1}{|S^{2N-1}|}\int_{S^{2N-1}}|\Lambda(P(z+r\eta ))|d\sigma(\eta)$$

$$\leq 2r^{-|\alpha|}\binom{N-1+|\alpha|/2}{N-1}\frac{\Gamma(|\alpha|/2)}{\Gamma (\alpha_1/2)\cdots\Gamma(\alpha_N/2)}\varphi_n(\mathbb{B}_N, q,\Lambda ,r),$$

$$\left\|D^\alpha P(z)\right\|\leq 2r^{-|\alpha|}\binom{N-1+|\alpha|/2}{N-1}\frac{\Gamma(|\alpha|/2)}{\Gamma (\alpha_1/2)\cdots\Gamma(\alpha_N/2)}\frac{1}{|S^{2N-1}|}\int_{S^{2N-1}}||P(z+r\eta )||d\sigma(\eta),$$
$$\leq 2r^{-|\alpha|}\binom{N-1+|\alpha|/2}{N-1}\frac{\Gamma(|\alpha|/2)}{\Gamma (\alpha_1/2)\cdots\Gamma(\alpha_N/2)}\varphi_n(\mathbb{B}_N, q,r),$$ where

$$ \varphi_n (\mathbb{B}_N, q,r):=\sup\{||P(z+\zeta)||:\ \zeta\in r\mathbb{B}_N,\ \deg P\leq n,||P(z)||\leq 1\},$$
$$ \varphi_n (\mathbb{B}_N, q,\Lambda,r):=\sup\{|\Lambda(P(z+\zeta))|:\ \zeta\in r\mathbb{B}_N,\ \deg P\leq n,||P(z)||\leq 1\}.$$
\end{exa}

\begin{rem} If $q$ is the supremum norm with respect to a compact $K\subset\mathbb{C}^N$ then in definitions of $\varphi_n$ we shall replace $q$ by $K$ and if $z_0\in K,\Lambda (P(z))=P(z_0)$, then we shall replace $\Lambda$ by $z_0$.
\end{rem}

\begin{exa}

\end{exa} Let $K=\mathbb{D},z_0\in\mathbb{T},\ \Lambda (P(z))=P(z_0),\ ||P||_{\mathbb{D}}=1,\ 1\leq \deg P\leq n$. Then
$$\int_0^{2\pi}|P(z+re^{i\theta})|^2\frac{d\theta}{2\pi}=\sum\limits_{k=1}^n\left(\frac{1}{k!}\right)^2|P^{(k)}(z_0)|^2r^{2k}\leq\sum\limits_{k=0}^n\left(\frac{1}{k!}\right)^2(n(n-1)\cdots (n-k+1))^2r^{2k}$$
$$=\sum\limits_{k=0}^n\binom{n}{k}^2r^{2k}$$ with equality if $P(z)=z^n$. (Here we use Bernstein inequality for derivative of polynomials on the unit circle). Hence
$$\varphi_n(\mathbb{D},z_0,2,r)=\left(\sum\limits_{k=0}^{n}\binom{n}{k}^2r^{2k}\right)^{1/2}$$ and 
$$\varphi_n(\mathbb{D},\mathbb{D},2,r)=\left(\sum\limits_{k=0}^{n}\binom{n}{k}^2r^{2k}\right)^{1/2}\leq (1+r)^n.$$ Moreover, 
$$\inf\limits_{r>0}r^{-l}\varphi_n(\mathbb{D},\mathbb{D},2,r)\leq \left(1/\binom{n}{l}^{1/l}\right)^{-l}\left(\sum\limits_{k=0}^{n}\binom{n}{k}^2(1/\binom{n}{l}^{1/l})^{2k}\right)^{1/2}$$
$$\leq \binom{n}{l}\left(1+1/\binom{n}{l}^{1/l}\right)^n\leq \binom{n}{l}(1+l/n)^n\leq e^l\binom{n}{l}.$$ Therefore 
$$\binom{n}{l}\leq \inf\limits_{r>0}r^{-l}\varphi_n(\mathbb{D},\mathbb{D},2,r)\leq e^l\binom{n}{l}.$$
\begin{exa}
Let $K=[-1,1],z_0\in[-1,1],\ \Lambda (P(z))=P(z_0),\ ||P||_{[-1,1]}=1,\ 1\leq \deg P\leq n$.
$$\int_0^{2\pi}|P(z_0+re^{i\theta})|^2\frac{d\theta}{2\pi}=\sum\limits_{k=1}^n\left(\frac{1}{k!}\right)^2|P^{k}(z_0)|^2r^{2k}\leq\sum\limits_{k=0}^n\left(\frac{1}{k!}\right)^2T_n^{(k)}(1)^2r^{2k}=\sum\limits_{k=0}^n(T_n^{(k)}(1)/k!)^2r^{2k}.$$ 
Hence 
$$\sup\left\{\int_0^{2\pi}|P(z_0+re^{i\theta})|^2\frac{d\theta}{2\pi},\ z_0\in [-1,1]\right\}\leq \sum\limits_{k=0}^n(T_n^{(k)}(1)/k!)^2r^{2k}$$ with equality for $P(z)=T_n(z)$. This gives equality
$$\varphi_n(\mathbb{D},[-1,1],2,r)=\left(\sum\limits_{k=0}^n(T_n^{(k)}(1)/k!)^2r^{2k}\right)^{1/2}\leq T_n(1+r).$$ Moreover,
$$\inf\limits_{r>0}r^{-l}\varphi_n(\mathbb{D},[-1,1],2,r)$$
$$\leq \left(1/(T_{n}^{(l)}(1)/l!)^{1/l}\right)^{-l}\left(\sum\limits_{k=0}^n(T_n^{(k)}(1)/k!)^2(1/(T_{n}^{(l)}(1)/l!)^{1/l})^{2k}\right)^{1/2}$$
$$\leq \frac{T_{n}^{(l)}(1)}{l!}T_n\left(1+(1/(T_{n}^{(l)}(1)/l!)^{1/l}\right)=\frac{T_{n}^{(l)}(1)}{l!}g\left(h\left(1+(1/(T_{n}^{(l)}(1)/l!)^{1/l}\right)^n\right)$$
$$=\frac{T_{n}^{(l)}(1)}{l!}g\left(h^n\left(1+\left(1/n2^l\binom{n+l-1}{2l}\right)^{1/l}\right)\right)$$
$$\leq \frac{T_{n}^{(l)}(1)}{l!}\left(1+(\sqrt{2}+\sqrt{6})\frac{l}{n+2l-1}\right)^n\leq e^{(\sqrt{2}+\sqrt{6})l}\frac{T_{n}^{(l)}(1)}{l!}.$$ Therefore 
$$\frac{T_{n}^{(l)}(1)}{l!}\leq \inf\limits_{r>0}r^{-l}\varphi_n(\mathbb{D},[-1,1],2,r)\leq e^{(\sqrt{2}+\sqrt{3})l}\frac{T_{n}^{(l)}(1)}{l!}.$$
\end{exa}
\begin{rem} In two above examples we have obtained the following.

Let $M_n(K,l):=\sup\{||P^{(l)}||_K:\ \deg P\leq n,||P||_K=1\}$. Then 
$$\varphi_n(\mathbb{D},K,2,r)\leq \left(\sum\limits_{k=0}^n\left(\frac{M_n(K,k)}{k!}\right)^2r^{2l}\right)^{1/2},$$
$$\frac{M_n(K,l)}{l!}\leq \inf\limits_{r>0}r^{-l}\varphi_n(\mathbb{D},K,2,r)\leq e(K)^l\frac{M_n(K,l)}{l!}$$ with
$$e(K)\leq \sup\limits_{n\geq  1}\sup\limits_{1\leq l\leq n}\left(\sum\limits_{k=0}^n\left(\frac{M_n(K,k)^{1/k}}{M_n(K,l)^{1/l}}\frac{(l!)^{1/l}}{(k!)^{1/k}}\right)^{2k}\right)^{1/2l}$$
$$\leq  \sup\limits_{n\geq  1}\sup\limits_{1\leq l\leq n}\left(\sum\limits_{k=0}^n\left(\frac{M_n(K,k)^{1/k}}{M_n(K,l)^{1/l}}\frac{(l!)^{1/l}}{(k!)^{1/k}}\right)^{k}\right)^{1/l}$$
\end{rem}
\begin{rem}
We can  repeat constructions from the above examples by considering another norms in $\mathbb{C}^N$, for example $||z||_p=\left(|z_1|^p+\dots+|z_N|^p\right)^{1/p}$. We shall consider below the  general case.
\end{rem}
\begin{defin} Fix  a norm $q(P)=||P||$ in $\mathbb{P}(\mathbb{C}^N)$, a circular and absorbing set $\mathbb{B}\subset\mathbb{C}^N$ (for any compact $C$ there exists an $r>0$ such that $K\subset [0,r]\mathbb{B}$) and a linear functional $\Lambda\in X_{q,n}^*$ with $||\Lambda||^*=1$. Then for any $r\geq 0$ define 

$$\varphi_n(\mathbb{B},q,\Lambda ,r):=\sup\{ |\Lambda (P(z+\zeta))|:\ \zeta\in r\mathbb{B},\deg P\leq n,||P(z)||\leq 1\},$$

$$\varphi(\mathbb{B},q,\Lambda ,r):=\sup\limits_{n\geq 1}\varphi_n(\mathbb{B},q,\Lambda ,r)^{1/n},$$
$$\ v(\mathbb{B},q,\Lambda ,r):=\log \varphi(\mathbb{B},q,\Lambda ,r),\  u(\mathbb{B},q,\Lambda ,t):= v(\mathbb{B},q,\Lambda ,e^t),t\in\mathbb{R},$$

$$\varphi_n(\mathbb{B},q,r):=\sup\{ ||P(z+\zeta)||:\ \zeta\in r\mathbb{B},\deg P\leq n,||P(z)||\leq 1\},$$

$$\varphi(\mathbb{B},q,r):=\sup\limits_{n\geq 1}\varphi_n(\mathbb{B},q,r)^{1/n},$$
$$v(\mathbb{B},q,r):=\log \varphi(\mathbb{B},q,r),\ u(\mathbb{B},q,t):=v(\mathbb{B},q,e^t),t\in\mathbb{R}.$$
\end{defin}

\begin{defin} If a norm $q(P)=||P||$ in $\mathbb{P}(\mathbb{C}^N)$ is fixed then define
$$M_n(q,\alpha ):=\sup\{ ||D^{\alpha} P||:\ \deg P\leq n,\ ||P||=1\},\ $$
$$e(q):=\sup\limits_{n\geq  1}\sup\limits_{1\leq |\alpha|\leq n}\left(\sum\limits_{|\beta|\leq n}\left(\frac{M_n(q,\beta)^{1/|\beta|}}{M_n(q,\alpha)^{1/|\alpha|}}\frac{(\alpha!)^{1/|\alpha|}}{(\beta!)^{1/|\beta |}}\right)^{|\beta|}\right)^{1/|\alpha|}.$$ 
\end{defin}
\medskip
\begin{rem} \begin{itemize} \ \\

\item If we define $ M_n^\bullet(q,k):=\sup\{\|D^\alpha P\|^{1/|\alpha|} \ : \ |\alpha|\le k, \ \deg P\le n, \ \|P\|=1\}$ then one cane easily check that $  M^\bullet_n(q,k) = \max_{j=1,...,N} M_n(q,e_j) .$

\item If $e(q)<\infty$ then 
$$\max\limits_{|\alpha |=k}\left(\frac{M_n(q,\alpha)}{\alpha!}\right)^{1/k}\leq \inf\limits_{r>0}r^{-k}\varphi_n(\mathbb{B}_N,q,r)\leq e(q)^{k}\max\limits_{|\alpha |=k}\left(\frac{M_n(q,\alpha)}{\alpha!}\right)^{1/k}.$$ Let us note that in the case $q(P)=||P||_{\mathbb{D}\cup\{2\}}$ we have $e(q)=\infty$. Thus condition $e(q)<\infty$ gives a some restriction. Here we can ask that condition $e(q)<\infty$ is equivalent to exits a constant $C$ such that $\max\limits_{|\alpha |=k}\left(\frac{M_n(q,\alpha)}{\alpha!}\right)^{1/k}\leq \inf\limits_{r>0}r^{-k}\varphi_n(\mathbb{B}_N,q,r)\leq C^{k}\max\limits_{|\alpha |=k}\left(\frac{M_n(q,\alpha)}{\alpha!}\right)^{1/k}.$ 

The second question is that condition $\sup\limits_{n\geq 2} \max\limits_{1\leq j\leq N}\log \frac{M_n(q,e_j)}{\log n}<\infty$ is necessary to satisfy the condition $e(q)<\infty$.
\end{itemize}

\end{rem}
\begin{pro} Let $\mathbb{B}_1$ and $\mathbb{B}_2$ be two unit closed balls in $\mathbb{C}^N$
such that $A_1\mathbb{B}_1\subset \mathbb{B}_2\subset A_2\mathbb{B}_1$. Then
$$A_1^k\inf\limits_{r>0}r^{-k}\varphi_n(\mathbb{B}_1,q,r)\leq \inf\limits_{r>0}r^{-k}\varphi_n(\mathbb{B}_2,q,r)\leq A_2^k\inf\limits_{r>0}r^{-k}\varphi_n(\mathbb{B}_1,q,r).$$
\end{pro}
\medskip

Applying arguments from \cite{BaranCiez1} one can prove the following important facts.

\begin{thm} The following functions are convex functions on $\mathbb{R}$ (possibly some of them are equal to $+\infty$):
 $$\log\varphi_n(\mathbb{B},q,\Lambda,e^t),\ \log\varphi (\mathbb{B},q,\Lambda,e^t),\ \log\varphi_n(\mathbb{B},q,e^t),\ \log\varphi (\mathbb{B},q,e^t).$$

\end{thm}
\medskip

\begin{rem} Since $\log\varphi_n(\mathbb{B},q,e^t)$ is a convex function,
 we get inequality
$$\varphi_n(\mathbb{B},q,rs)\leq \varphi_n(\mathbb{B},q,r^p)^{1/p}\varphi_n(\mathbb{B},q,s^q)^{1/q},\
\frac{1}{p}+\frac{1}{q}=1.$$

\end{rem}
\medskip

As a direct consequence of this theorem and known properties of convex functions we get an important properties (c.f. \cite{BaranCiez1}).
\medskip

\begin{cor} The following functions if are finite then are continuous and increasing on $(0,\infty)$:

$$\log\varphi_n(\mathbb{B},q,\Lambda,r),\ \log\varphi (\mathbb{B},q,\Lambda,r),\ \log\varphi_n(\mathbb{B},q,r),\ \log\varphi (\mathbb{B},q,r).$$

\end{cor}
\medskip

Applying known (but still dificult to prove) we get one of  reasons that introduced notions can be helpful.

\begin{cor} If $ \log\varphi (\mathbb{B},q,\Lambda,r)$ or $ \log\varphi (\mathbb{B},q,r)$ is finite then this function is differentiable except possibly countable set of points and is twice differentiable almost everywhere (Alexandrov's theorem).

\end{cor}

\bigskip


\section{Radial modifications of Siciak's extremal function.}
\bigskip

If $E\subset\mathbb{C}^N$ is a compact set then Siciak's extremal
function $\Phi_E(z)=\Phi (E,z)$ is usually defined as
$$\Phi (E,z):=\sup\{ |P(z)|^{1/\deg P}:\ \deg P\geq 1,\ ||P||_E\leq
1\},\ z\in\mathbb{C}^N.$$ In connection with $\Phi (E,z)$ there
are also considered functions $\Phi_n(E,z)$, where
$$\Phi_n(E,z)=\sup\{ |P(z)|:\ \deg P\leq n,\ ||P||_E\leq
1\},\ z\in\mathbb{C}^N.$$ There is known that (c.f. 3.2 in
\cite{SiciakKukuryku}) for all $z\in\mathbb{C}^N$
$$\Phi(E,z)=\sup\limits_{n\geq
1}\Phi_n(E,z)^{1/n}=\lim\limits_{n\rightarrow\infty}\Phi_n(E,z)^{1/n}.$$
The $L-capacity$ is defined as $C(E):=\liminf_{z\rightarrow \infty}||z||_2/\Phi(E,z)$ (cf. \cite{Klimek},\cite{SiciakAnnalesy},\cite{SiciakKukuryku}), which is Choquet capacity \cite{Kolodziej} and has product property $C(E\times F)=\min (C(E),C(F))$ \cite{BaranCiez1}.

Analogously we can define $C_{\nu}(E):=\liminf_{z\rightarrow\infty}\nu (z)/\Phi (E,z)$, where $\nu$ is a norm in $\mathbb{C}^N$. We refer to \cite{BaranCiez1} for examples, where $C_\nu (E)$ is explicitely computed. In the case $\nu(z)=||z||_p$ we also have product property: $C_{\nu} (E\times F)=\min (C_{\nu_1}(E),C_{\nu_2}(F))$, where $\nu_j(z_j)=||z_j||_p$.

Now  for $r\geq 0$ define

$$\varphi_n (r)=\varphi_n (E,r):=\sup\{ |P(z+\zeta )|:\ z\in
E,\ ||\zeta ||_2\leq r,\ \deg P\leq n,\ ||P||_E\leq 1\}$$

$$=\varphi_n(\mathbb{B}_N,E,r)=\sup\{ \Phi_n (E,z+\zeta):\ z\in E,||\zeta||_2\leq r\},$$

$$\varphi (r)=\varphi (E,r):=\sup\limits_{n\geq 1}\varphi_n (E,r)^{1/n}=\sup\{ \Phi (E,z+\zeta):\ z\in E,||\zeta||_2\leq r\},$$

 and
$$v(r)=v(E,r)=\log\varphi (E,r),\ v_n(r)=\log\varphi_n(E,r),\ r\geq 1,$$

$$u(t)=u(E,t)=v(e^t),\ u_n(t)=v_n(e^t),\ t\in\mathbb{R}.$$

An important tool in polynomial approximation theory plays the {\it homogeneous capacity}$\sigma (E)$ related to the {\it homogeneous Siciak extremal function} $\Psi (E,z)$:
$$\Psi (E,z)=\sup\limits_{n\geq 1}\Psi_n (E,z)^{1/n}=\lim\limits_{n\rightarrow\infty}\Psi_n (E,z)^{1/n},$$ where 
$$\Psi_n(E,z)=\sup\{ |P(z)|:\ P\ \text{\rm  homogeneous\ of degree}\ n,\ \ ||P||_E\leq
1\},\ z\in\mathbb{C}^N.$$
$$\sigma_\nu(E):=\liminf\limits_{z\rightarrow\infty}\frac{\nu(z)}{\Psi (E,z)}=\inf\limits_{\nu(z)=1}\frac{1}{\Psi (E,z)}=\frac{1}{\sup\limits_{\nu(z)=1}\Psi (E,z)}.$$
There is known(c.f. \cite{Korevaar1} - \cite{Korevaar4}) the following description of homogeneus capacity in the case $\nu (z)=||z||_\infty$: $$\sigma_\nu (E)=\inf\limits_{n\geq 1}\beta_n(E)=\lim\limits_{n\rightarrow\infty}\beta_n(E),$$
where 
$$\beta_n(E)=\inf\limits_{|\alpha|=n}\inf_{b_\beta\in\mathbb{C}} \|
z^\alpha-\sum\limits_{|\beta|=n,\beta\neq\alpha}b_\beta z^\beta\|_E^{1/n}.$$ The constants $\beta_n(E)$ are optimal in the following deep result, which was proved by Koreavaar refining earlier joint lemma with Wiegerinck (cf. \cite{Korevaar1} - \cite{Korevaar4}).
\medskip

\begin{pro}
 If $E\subset S^{N-1}\subset \mathbb{R}^N$ satisfy $\beta_n(E)>0$ then for any $f$ $C^\infty$ function on a neighborhood of some point $a\in\mathbb{R}^N$ one has inequality:
$$\max\limits_{|\alpha|=n}\binom{n}{\alpha}\left|D^{\alpha}f(a)\right|\leq \sup\limits_{y\in E}\left|\left(\frac{d}{dt}\right)^nf(a+ty)|_{t=0}\right|/\beta_n(E)^n\leq \sup\limits_{y\in E}\left|\left(\frac{d}{dt}\right)^nf(a+ty)|_{t=0}\right|/\sigma(E)^n.$$
\end{pro}

\begin{rem} The proposition fails in the most interesting case $E=\{e_1,\dots,e_N\}$, especially in the case of polynomials. However if we consider family of constants $M_n(q,\alpha)$ then probably the following is true:

if $q$ is an arbitrary norm in $\mathbb{P}(\mathbb{C}^N)$ then there exists a positive constant $a=a(q)$ such that 
$$\max\limits_{|\alpha|=k}M_n(q,\alpha)\leq a(q)^k\max\limits_{1\leq j\leq N}M_n(q,ke_j).$$

Let us note, as an example, that for $q(P)=||P||_{\mathbb{D}_N}$ one can take $a(q)=e^{N}$. Similarly, in the case $q(P)=||P||_{[-1,1]^N}$ we can put $a(q)=e^{2N}$.

\end{rem}
\medskip

\begin{exa}
\begin{itemize}
\item[]

\item[(1)] If $E=\{ z\in\mathbb{C}^N:\ ||z||\leq 1\}$ ($||z||$ is a norm in $\mathbb{C}^N$) then 
$\varphi (E,r)=1+r/C(E)$. By \cite{Monn}
 $$(dd^c\log \varphi (E,||z||_2)^N=\frac{1}{4}\left(\frac{1}{2}\right)^{N-1}\left(\frac{1}{||z||_2}\right)^{N}\left(\frac{1}{||z||_2+C(E)}\right)^{N+1}C(E),$$
 $$\lim_{z\rightarrow\infty}2^{N+1}||z||_2^{2N+1}(dd^c\log \varphi (E,||z||_2)^N=C(E).$$
\medskip

\item[(2)] If $E=\{ z\in\mathbb{R}^N\subset\mathbb{C}^N=\mathbb{R}^N+i\mathbb{R}^N+i\mathbb{R}^N:\ \nu(z)\leq 1\}$ then 
$\varphi (E,r)=h(1+r/(2C(E)))$, where $h(t)=t+\sqrt{t^2-1},\ t\geq 1$. (We also have $h(t)=g|_{[1,+\infty)}^{-1}(t),\ g(t)=\frac{1}{2}(t+1/t)=t-\hat{g}(t),\ \hat{g}(t)=\frac{1}{2}(t-1/t)$.) In this case we can calculate 
$$(dd^c\log \varphi (E,||z||_2)^N=\frac{1}{4}\left(\frac{1}{2}\right)^{N-1}\left(\frac{1}{||z||_2}\right)^{3N/2}\left(\frac{1}{||z||_2+4C(E)}\right)^{N/2+1}2C(E),$$
$$\lim_{z\rightarrow\infty}2^{N+1}||z||_2^{2N+1}(dd^c\log \varphi (E,||z||_2)^N=2C(E).$$
\medskip

\item[(3)] If $E$ is the closed unit ball in $\mathbb{C}^N$ with respect to a norm $n(z)=||z||$ then there is known (cf. \cite{SiciakAnnalesy},\cite{SiciakKukuryku}) that $\Psi(E,z)=||z||$ (while $\Phi (E,z)=\max(1,||z||)$), whence
$$C_\nu(E)=\sigma_{\nu}(E)=\frac{1}{\sup\limits_{\nu(z)=1}||z||}.$$
\medskip

\item[(4)] A situation is much more complicated if $E$ is a convex symmetric body in $\mathbb{R}^N$. There was known in the case $E$ is the unit Euclidean ball in $\mathbb{R}^N$ than $\Psi (E,z)=L(z)=\left(\frac{||z||_2^2+|z_1^2+\dots z_N^2|}{2}\right)^{1/2}$ is the Lie norm
(which gives $\sigma (E)=\frac{1}{\sqrt{2}}>\frac{1}{2}=C(E)$).

If $N>2$ a situation is quite unclear. But in the case $N=2$ there is known the following result \cite{Baran1} (cf. \cite{Baran2}):

Let $S$ be the unit ball with respect to a norm $N$ in
$\mathbb{R}^2$. If $u(t)=\log N(1,t)$ then
$$\Psi(S,(z_1,z_2))=|z_1|\exp \mathcal{P}u(z_2/z_1),$$
with $$ \mathcal{P}u(\zeta )=(\Im\zeta
)\frac{1}{\pi}\int_{-\infty}^\infty |\zeta
-t|^{-2}u(t)dt=\frac{1}{\pi}\int_{-\infty}^\infty u(ty+x)
\frac{dt}{1+t^2},$$ where $\zeta =x+iy,\ y\geq 0$.
\medskip

In particular, ff $N_m(x)=\left(x_1^{2m}+x_2^{2m}\right)^{1/(2m)}$ and $S_m=\{
x\in\mathbb{R}^2:\ N_m(x)=1\}$, then for all $z\in\mathbb{C}^2$,
$$\Psi(S_m,z)=\left[ \prod\limits_{j=1}^m\left(|z_1|^2-2\alpha_j
\Re (z_1\overline{z_2})+|z_2|^2+2 |\beta_j|\Im
(z_1\overline{z_2})|\right)^{1/2}\right]^{1/m},$$ where
$\zeta_j=\alpha_j+i\beta_j\in\root{2m}\of{-1},\ j=1,\dots ,m, $
with $\zeta_j\neq \overline{\zeta_k}$ for $j\neq k$.

If $N_\infty (x)=\max (|x_1|,|x_2|)$ and $S_\infty=\{
x\in\mathbb{R}^2:\ N_\infty(x)=1\}$, then for all
$z\in\mathbb{C}^2$,
$$\Psi (S_\infty,z)=\exp \left[
\int_0^{2\pi}\log\left(|z_1|^2-2\cos\theta  \Re
(z_1\overline{z_2})+|z_2|^2+2 |\sin\theta \Im
(z_1\overline{z_2})|\right)^{1/2}\frac{d\theta}{2\pi}\right] .$$
Since $S_1=\{x\in \mathbb{R}^2:\ |x_1|+|x_2|= 1\}=L^{-1}(S_\infty)
$, where $L(z_1,z_2)=(z_1-z_2,z_1+z_2)$, we get 
$$\Psi (S_1,z)=\Psi(S_\infty,L(z))$$
$$=\exp \left[
\int_0^{2\pi}\log\left(2|z_1|^2+2|z_2|^2-2\cos\theta  (|z_1|^2-|z_2|^2)+4|\sin\theta \Im (z_1\overline{z_2})|\right)
\frac{d\theta}{4\pi}\right].$$
\end{itemize}
\end{exa}
\bigskip

\begin{exa} 

\begin{itemize}
\item[•] \ \

\item[(1)] 
 $E=\{z\in\mathbb{C}:\ |z|\leq R\}$, $u(t)=\log\varphi (E,e^t)=\log (1+ce^t),\ c=\frac{1}{C(E)}=\frac{1}{R},\ \frac{u''(t)}{u'(t)}=\frac{1}{1+ce^t}$, $\lim_{t\rightarrow\infty}e^t \frac{u''(t)}{u'(t)}=\frac{1}{c}$.
 \medskip
 
 $\Delta \log\varphi (E,|z|)=\frac{c}{|z|(1+c|z|)^2}$, $\lim_{z\rightarrow\infty}|z|^3\Delta \log\varphi (E,|z|)=\frac{1}{c}$.
\medskip

\item[(2)] $E=[a,b]$, $
u(t)=\log h(1+ce^t),\ c=\frac{1}{2C(E)}=\frac{2}{b-a},$ 
$ \frac{u''(t)}{u'(t)}=\frac{1}{2+ce^t}.$ $\lim_{t\rightarrow\infty}e^t \frac{u''(t)}{u'(t)}=\frac{1}{c}$.
\medskip

$\Delta \log\varphi (E,|z|)=\frac{c^2}{((1+c|z|)^2-1)^{3/2}}$, $\lim_{z\rightarrow\infty}|z|^3\Delta \log\varphi (E,|z|)=\frac{1}{c}$.
\medskip

\item[(3)]
$E=\{ z\in\mathbb{C}:\ \Phi([-1,1],z)\leq R\},\ R\geq 1$, 
$u(t)=\log h(g(R)+e^t)-\log R,$
$ \frac{u''(t)}{u'(t)}=\frac{\frac{g(R)+1}{2}}{g(R)+1+e^t}+\frac{\frac{g(R)-1}{2}}{g(R)-1+e^t}.$ $\lim_{t\rightarrow\infty}e^t \frac{u''(t)}{u'(t)}=g(R)$.
\medskip

$\Delta \log\varphi (E,|z|)=\frac{\hat{g}(R)^2+g(R)|z|}{|z|((g(R)+|z|)^2-1)^{3/2}}$, $\lim_{z\rightarrow\infty}|z|^3\Delta \log\varphi (E,|z|)=g(R)$.

\item[(4)]
$E=[-1,1]\times \overline{\mathbb{D}}_{R}$

$$u(t)=\max \left(\log h(1+e^t),\log (1+e^t/R)\right)=\begin{cases} \log h(1+e^t),\ R\geq \frac{1}{2}\\
\log h(1+e^t),\ t<\log\left( \frac{2}{(1/R-1)^2-1}\right),\
0<R<\frac{1}{2},\\ \log (1+e^t/R),\ t>\log\left(
\frac{2}{(1/R-1)^2-1}\right)
\end{cases}$$

$$ \frac{u''(t)}{u'(t)}=\begin{cases}
\frac{1}{2+e^t},\ t<\log\left( \frac{2}{(1/R-1)^2-1}\right),\ 0<R<\frac{1}{2},\\
\frac{1}{1+e^t/R},\ t>\log\left( \frac{2}{(1/R-1)^2-1}\right)
\end{cases}$$
\medskip

\item[(5)]$E=\mathbb{D}\cup\{z_0\},\ z_0\not\in\mathbb{D}$. Then $\varphi (E,r)=|z_0|+r$, $r>0$ and $\varphi (E,0)=1$. Further, $\log\varphi (E,r)=\log (1+r/|z_0|)+\log |z_0|=\log\varphi (|z_0|\mathbb{D},r)+\log |z_0|$, whence $\frac{u''(t)}{u'(t)}=\frac{|z_0|}{|z_0|+r}$,

$$\Delta \log\varphi (E,|z|)=\frac{1/|z_0|}{|z|(1+|z|/|z_0|)^2}, \ \lim_{z\rightarrow\infty}|z|^3\Delta \log\varphi (E,|z|)=|z_0|=\lim\limits_{t\rightarrow\infty}e^t\frac{u''(t)}{u'(t)}.$$
Let us note (for $z_0=2$) that applying the  formula $\varphi (E,r)=2+r$ and considering polynomials $P_n(z)=(z-a_n)z^{n-1}$, where $a_n=3\frac{2^{n-1}}{2^{n-1}+1}-1$ we get bounds 
$$(n-1)\cdots (n-k+1)2^{-k}(n+\frac{k}{3}(2^n-1))\leq M_n(E,k)\leq e^k2^{n-k}n^k.$$
\medskip


\end{itemize}
\end{exa}
\bigskip

Let us recall mentioned above David Monn result from \cite{Monn} (it is only one paper published by this mathematician).

\begin{pro} If $U$ is the $C^2$ plurisubharmonic funcion on $\mathbb{C}^N$ that is radial ($U(z)=u(||z||_2)$ with $u\in C^2(\mathbb{R}_+$) then 
$$(dd^cU)^n(z)=\frac{1}{4}\left(\frac{u'(||z||_2}{2||z||_2}\right)^{N-1}\left(u''(||z||_2)+\frac{1}{||z||_2}u'(||z||_2)\right).$$
\end{pro}
\begin{cor} If $\lim_{r\rightarrow \infty}ru'(r)=1$ then 
$$\lim\limits_{z\rightarrow\infty}2^{N+1}||z||_2^{2N+1}(dd^cU)^n(z)=\lim_{t\rightarrow\infty}e^t\frac{v''(t)}{v'(t)},\ v(t)=u(e^t).$$

\end{cor}

\begin{rem} Let use notice some observations in the one dimensional case.

\begin{itemize}

\item
We have $\varphi_1(E,r)=1+M_1(E)r$, where $$M_1(E)=\sup\{
||P'||_E:\ P\in \mathbb{P}_1( \mathbb{C}),\ ||P||_E=1\}.$$ Hence
$$(1+M_1(E)r)^n\leq \varphi_n(E,r),\ n\geq 1, $$ or equivalently
 $(1+r)^n\leq \varphi_n(E,r/M_1(E))$. As an
application we get the following:

\item[] For all $r,s\geq 0$
$$ \varphi_n(E,rs)\leq \varphi_n(E,r)\varphi_n(E,s/M_1(E)).$$
Moreover, if $r\geq 1/M_1(E)$ then
$$ \varphi_n(E,r+s)\leq \varphi_n(E,r)\varphi_n(E,s).$$

\begin{proof} As a consequence of Bernstein's inequality we get

$$ \varphi_n(E,rs)\leq \varphi_n(E,r)\max (1,s)^n
\leq \varphi_n(E,r)(1+s)^n\leq
\varphi_n(E,r)\varphi_n(E,s/M_1(E)).$$ Analogously
$$ \varphi_n(E,r+s)\leq \varphi_n(E,r)(1+s/r)^n\leq \varphi_n(E,r)(1+M_1(E)s)^n\leq
\varphi_n(E,r)\varphi_n(E,s).$$
\end{proof}

\item  Let $E$ be a Bernstein set, that is $$M(E,1)=\sup\left\{ \frac{||P'||_E}{\deg
P},\ \deg P\geq 1,\ ||P||_E=1\right\}<+\infty.$$ Then
$$(1+M_1(E)r)^n\leq \varphi_n(E,r)\leq (1+M(E,1)r)^n$$ and
$$\varphi_n(E,r+s)\leq
\varphi_n(E,(M(E,1)/M_1(E))r)\varphi_n(E,(M(E,1)/M_1(E))s).$$

\item  Define (cf. \cite{MirekAgnieszkaBeataPawel})
$$ ||P||_r=\sum\limits_{k=0}^\infty \frac{1}{k!}||P^{(k)}||_Er^k,\ r\geq 0, \ ||P||_0=||P||_E.$$ Then 
$$ \sup\limits_{|\zeta|\leq r}||P(x+\zeta )||_E\leq ||P||_r\leq (\deg P+1) \sup\limits_{|\zeta|\leq r}||P(x+\zeta )||_E.$$
Hence
$$\varphi_n(E,r)\leq \sup\{ ||P||_r:\ P\in\mathbb{P}_n( \mathbb{C}),||P||_0\leq 1\}\leq (n+1)\varphi_n(E,r)$$ and therefore
$$\varphi (E,r)=\lim_{n\rightarrow\infty}\sup\{ ||P||_r:\ P\in\mathbb{P}_n( \mathbb{C}),||P||_0\leq 1\}^{1/n}.$$
\end{itemize}
\end{rem}

\bigskip

 If
$E$ is a compact subset of $ \mathbb{C}^N$ with $C(E)>$ then there
is known (c.f. \cite{BaranCiez1}) that $u(t)=\log\varphi (E,e^t)$ is a convex
increasing function and $\Lambda (t)=u(t)-t$ is a (convex)
decreasing one with $\Lambda (t) \searrow -\log C(E)$. In
particular
\medskip

$\lim\limits_{r\rightarrow\infty}
\frac{\varphi (E,r)}{r}=\frac{1}{C(E)}$ and $
\frac{\varphi (E,r)}{r}\searrow\frac{1}{C(E)}$. \bigskip








\bigskip




\begin{pro} Assume that $v_n(E,r)$ is finite for $r>0$. Then there
is a positive constant $C_n(E)$ such that
$$v_n(E,r)-\log r\searrow -\log C_n(E)$$ and thus $C_n(E)=\lim\limits_{r\rightarrow +\infty}
\frac{r}{\varphi_n(E,r)}=\sup\limits_{r>0}\frac{r}{\varphi_n(E,r)}$,
which implies $\varphi_n(E,r)\geq \frac{r}{C_n(E)}$.

\end{pro}
\bigskip

\section{A radial extremal function related to a norm in $
\mathbb{P}( \mathbb{C})$} \ \\
\medskip

\bigskip

\begin{pro} Assume that $v_n(q,r)$ is finite for $r>0$. Then there
is a positive constant $C_n(q)$ such that
$$v_n(q,r)-\log r\searrow -\log C_n(q)$$ and thus $C_n(q)=\lim\limits_{r\rightarrow +\infty}
\frac{r}{\varphi_n(q,r)}=\sup\limits_{r>0}\frac{r}{\varphi_n(q,r)}$,
which implies $\varphi_n(q,r)\geq \frac{r}{C_n(q)}$. Moreover
$$t_n(q)^{1/n}\geq C_n(q),\ n\geq 1,\   t(q)\geq \min\limits_{n\geq
1}C_n(q)\ \ \ \ \text{\rm and}\ \ \ t(q)\geq C(q).$$ Here
$C(q):=\lim\limits_{r\rightarrow\infty} \frac{r}{\varphi (q,r)}$
or equivalently $-\log
C(q)=\lim\limits_{r\rightarrow\infty}(v(q,r)-\log r)$ and more
precisely  $v(q,r)-\log r\searrow -\log C(q)$.

\end{pro}
\bigskip

\begin{proof}

Fix $n\geq 1$ and a polynomial $P\in \mathbb{P}_n( \mathbb{C}),\
||P||=1$, and a continuous functional $l$ with $||l||^*=1$.
Consider the function
$$g(\zeta )=\frac{1}{n}\log |l(P(x+\zeta ))-\log |\zeta |\in
\mathrm{SH}( \mathbb{C}\setminus \mathbb{D}_{r_0}),\ r_0>0.$$
Since $g$ is bounded from above, we have, by the maximum principle
for subharmonic functions, the inequality $g(\zeta )\leq
\max\limits_{|\zeta|=r_0}g(\zeta )$. Taking the supremum we get
the bound
$$v_n(q,r)\leq v_n(q,r_0)+\log r-\log r_0,\ r\geq r_0.$$ Now consider the
function $\psi (t)=v_n(q,e^t)-t$. It is a convex function that is
bounded from above which implies
$\liminf\limits_{t\rightarrow+\infty}\frac{1}{t}\psi (t)\leq 0$
and by Lemma $\psi $ is a decreasing function. In particular the
limit $\lim\limits_{r\rightarrow\infty}(v_n(q,r)-\log r)=:-\log
C_n(q)$ exists and $-\log C_n(q)=\inf\limits_{r>0}(v_n(q,r)-\log
r)$.

Similarly, assuming $v(q,r_0)$  is finite for an $r_0>0$ and
applying analogous arguments we get existence of the limit
$\lim\limits_{r\rightarrow\infty}(v(q,r)-\log r)=:-\log C(q)$ and
$-\log C(q)=\inf\limits_{r>0}(v(q,r)-\log r)$.

Now let $T_n(q)=T_n(q,\cdot )$ be $n-th$ Chebyshev polynomial for
$q$: $T_n(q)$ is a monic polynomial of degree $n$ such that
$$t_n(q)=||T_n(q)||=:\inf\{ ||P_n||:\ P_n\ \text{\rm is a monic
polynomial of degree}\ n\}.$$ Then
$$ \sup\limits_{ |\zeta |=r}\log\left( \frac{||T_n(q,x+\zeta
)||}{||T_n(q)||}\right)^{1/n}\leq v_n(q,r)$$ which easily gives
$-\log C_n(q)\geq -\log ||T_n(q)||^{1/n}$. Analogously we get
inequality $-\log C(q)\geq -\log ||T_n(q)||^{1/n}$ and therefore
$t(q)\geq C(q)$.
\end{proof}

\bigskip
\section{Markov's inequality in $ \mathbb{C}$}\ \\

Let $E$ be a compact subset of $ \mathbb{C}$. Applying Cauchy's
integral formula one can easily prove the following.
\medskip

\begin{pro} If $P\in \mathbb{P}_n( \mathbb{C}),\ n\geq 1$ with
$||P||_E=1$ then
$$||P'||_E\leq \inf\limits_{r>0} \frac{1}{r}\exp (nv_n(E,r))\leq \inf\limits_{r>0} \frac{1}{r}\exp
(nv(E,r)).$$

\end{pro}
\bigskip

\begin{pro} Assume that $ E\in \mathcal{AM}(M,m)$, which means that for
an arbitrary $P\in\mathbb{P}( \mathbb{C})$ the following A. Markov
type inequality is satisfied: $||P'||_E\leq M(\deg P)^m||P||_E$
(here $M>0,m\geq 1$ are constants). Then we have the following
bounds
$$v_n(E,r)\leq Mn^{m-1}r,\ r\geq 0.$$

\end{pro}

\begin{pro} Assume that $ E\in \mathcal{AM}(M,m)$. Then
$$v_n(E,1)\leq  M+(m-1)\log n.$$

\end{pro}

\begin{proof} Fix an $x\in E$ and  $P\in \mathbb{P}_n(
\mathbb{C})$. Consider the function
$$g(\zeta )=\frac{1}{n}\log |P(x+\zeta )-\log |\zeta |\in
\mathrm{SH}( \mathbb{C}\setminus \mathbb{D}_{r_0}),\ r_0>0.$$
Since $g$ is bounded from above, we have by the maximum principle
for subharmonic functions, the inequality $g(\zeta )\leq
\max\limits_{|\zeta|=r_0}g(\zeta )$. Taking the supremum we get
the bound
$$v_n(E,r)\leq v_n(E,r_0)-\log r_0+\log r,\ r\geq r_0$$ and for
$r_0=\left( \frac{1}{n}\right)^{m-1}$ we obtain $v_n(E,1)\leq
M+(m-1)\log n$.

\end{proof}
\bigskip

\bigskip

\begin{pro} Assume that $ q\in \mathcal{AM}(M,m)$, which means that for
an arbitrary $P\in\mathbb{P}( \mathbb{C})$ the following A. Markov
type inequality is satisfied: $||P'||\leq M(\deg P)^m||P||$ (here
$M>0,m\geq 1$ are constants). Then we have the following bounds
$$v_n(q,r)\leq Mn^{m-1}r,\ r\geq 0.$$

\end{pro}
\bigskip

\begin{pro}

Assume that $ q\in \mathcal{VM}(M,m)$, which means that for an
arbitrary $P\in\mathbb{P}( \mathbb{C})$ the following V. Markov
type inequality is satisfied: $||P'^{(k)}||\leq
M^k\left(\frac{1}{k!}\right)^{m-1}(\deg P)^{km}||P||$ (here
$M>0,m\geq 1$ are constants). Then we have the following bounds
$$v_n(q,r)\leq mM^{1/m}r^{1/m},\ r\geq 0.$$

\end{pro}
\bigskip

\begin{proof} Applying Taylor's expansion to $P\in \mathbb{P}_n( \mathbb{C})$ with $||P||=1$ we can write
$$ ||P(x+\zeta )||\leq \sum\limits_{k\leq n}
\frac{1}{k!}||P^{(k)}(x)|||\zeta|^k\leq \sum\limits_{k\leq n}
\frac{1}{k!}M^k \left(\frac{1}{k!}\right)^{m-1}n^{km}|\zeta|^k$$
$$\leq \left(\sum\limits_{k\leq
n}\frac{1}{k!}(M^{1/m}n|\zeta|^{1/m})^k\right)^m\leq\left[ \exp
(M^{1/m}n|\zeta |^{1/m})\right]^m=\exp (mM^{1/m}n|\zeta
|^{1/m}),$$ which implies
$$v_n(q,r)\leq mM^{1/m}r^{1/m}.$$
\end{proof}
\medskip

We shall write $q\in HCP(\gamma ,B)$ ($\gamma ,B$ positive
constants) if inequality $v(q,r)\leq Br^\gamma$ holds for an
arbitrary $r>0$.
\bigskip

\begin{thm} Let $q$ be a fixed norm in $ \mathbb{P}( \mathbb{C})$.
Then we have implications
$$q\in VM(m,M)\ \Rightarrow\ q\in HCP( \frac{1}{m},mM^{1/m}),$$
$$ q\in HCP(\gamma ,B)\ \Rightarrow \ q\in VM(
\frac{1}{\gamma},(\gamma eB)^{1/\gamma}).$$

Moreover, if $q\in HCP(\gamma ,B)$ then $$\gamma eB\cdot
C(q)^\gamma \geq 1.$$

Hence, if $q\in VM(m,M)$ then $$C(q)\geq \frac{1}{M}e^{-m}.$$

\end{thm}
\medskip

\begin{proof} The proof of the first implication was done. Assume
$q\in HCP(\gamma ,B)$ and take $P\in \mathbb{P}_n( \mathbb{C}),\
||P||=1$. Then
$$||P^{(k)}||\leq k!\inf\limits_{r>0} \frac{1}{r^k}\exp
(nBr^{\gamma})=k!\inf\limits_{r>0} \exp (nBr^{\gamma}-k\log r).$$
The minimum is attained for $r=(k/nB\gamma )^{1/\gamma}$ which
gives inequality
$$||P^{(k)}||\leq k!\left( \frac{n}{k}\right)^{k/\gamma}(B\gamma
e)^{k/\gamma}\leq (B\gamma e)^{k/\gamma}\left(
\frac{1}{k!}\right)^{1/\gamma-1}n^{k/\gamma}.$$

Assume again $q\in HCP(\gamma ,B)$. Since $v(q,r)$ is continuous
then $v(q,[0,+\infty ))=[0,+\infty )$ and we can take a positive
$r$ such that  $v(q,r)=\frac{1}{\gamma}$. Now
$$C(q)\geq \frac{r}{\exp v(q,r)}=\frac{r}{e^{1/\gamma}}=\frac{r}{(\gamma
ev(q,r))^{1/\gamma}}$$
$$\geq \frac{r}{(\gamma eBr^\gamma )^{1/\gamma}}=\frac{1}{(\gamma
eB)^{1/\gamma}}.$$

\end{proof}

\section{Extremal functions related to Ple\'sniak's properties.}

\bigskip
\begin{defin} Fix a norm $q$ in $\mathbb{P}(\mathbb{C}^N)$ we define a family of
 extremal radial functions 
 $$\mathcal{R}_k(q,r):=\sup\limits_{n\geq 1}\varphi_n\left(q,r(k!/M_n(q,k))^{1/k}\right)^{1/k}.$$
 and 
 
$$\mathcal{R}(q,r):=\sup\limits_{n\geq 1}\sup\limits_{1\leq k\leq n}\varphi_n\left(q,r(k!/M_n(q,k))^{1/k}\right)^{1/k}.$$

As an example consider $E=\mathbb{D}$. Since $\varphi_n(E,r)=(1+r)^n$ we get 
$$\mathcal{R}_1(E,r)=\sup\limits_{n\geq 1}(1+r/n)^n=e^r,$$
and since $\varphi_n(E,r/\binom{n}{k}^{1/k})^{1/k}\leq (1+rk/n)^{n/k}\leq e^r$, we obtain
$$\mathcal{R}(\mathbb{D},r)=\mathcal{R}_1(\mathbb{D},r)=e^r,\ r\geq 0.$$

In the case $E=[-1,1]$ we can estimate
$$\cosh\sqrt{2r}\leq \mathcal{R}_1(E,r)\leq e^{\sqrt{2r}},\ \mathcal{R}(E,r)=e^{\sqrt{2r}}.$$

Let $E_R=\{ z\in\mathbb{C}:\ |h(z)|\leq R\}=\{
z\in\mathbb{C}:\ \Phi ([-1,1],z)\leq R\},\ R>1.$

 Then $\varphi
(E_R,r)=h(g(R)+r)/R$. One can check that
$$\sup\limits_{n\geq 1}\varphi (E_R,r/n)^n=\lim\limits_{n\rightarrow\infty}\varphi (E_R,r/n)^n
=e^{2r/(R-1/R)}=e^{r/\sqrt{g^2(R)-1}}.$$ Hence
$$ \mathcal{P}_1(E_R,r)\leq e^{r/\sqrt{g^2(R)-1}}.$$

\end{defin} 
  
  \begin{rem}
 We can  define 
$$\widetilde{\varphi}_n(q,r)=\sum\limits_{k=0}^n\frac{1}{k!}M_n(E,k)r^k. $$

We have $\varphi_n(q,r)\leq \widetilde{\varphi}_n(q,r) $ and thus

$$  \varphi_n\left(q,r(k!/M_n(q,k))^{1/k}\right)^{1/k}\leq  
 \left( \sum\limits_{l=0}^n\frac{(k!)^{l/k}}{l!}\left(M_n(q,l)^{1/l}/M_n(q,k)^{1/k}\right)^{l}r^l \right)^{1/k}.$$
 Hence, if $\sup\limits_{n\geq 1}\sup\limits_{1\leq k,l\leq n}M_n(q,l)^{1/l}/M_n(q,k)^{1/k}:=a(q)<\infty$ then 
 $\mathcal{R}(q,r)\leq e^{a(q)r}$.\end{rem}
 \medskip
 
\begin{defin}

Let us recall that $q\in \mathcal{AM}(m,M)$ iff $M_n(q,1)\leq Mn^m$ and $q\in \mathcal{VM}(m,M)$ iff $M_n(q,k)\leq M^kn^{km}/(k!)^{m-1}$.  This is a motivation to consider {\it Ple\'sniak's extremal functions}

$$\mathcal{P}_m (q,r):=\sup_{n\geq
1}\varphi_n\left(q,\frac{r}{n^m}\right),$$

$$\mathcal{B}_\alpha ( q,r):=\sup_{n\geq
1}\sup\limits_{k\geq
1}\varphi_n\left(q,r\left(\frac{k}{n}\right)^\alpha\right)^{1/k}.$$

Let us observe that
$$ \mathcal{P}_m(q,r)\leq \mathcal{B}_m
(q,r)= \sup\limits_{k\geq 1} \mathcal{P}_m(q,rk^m)^{1/k}.$$
\end{defin}

\begin{rem}  
If $E\in\mathcal{M}(E,m,M)$ then 
$$ \widetilde{\varphi}_n(E,r)\leq \sum\limits_{k=0}^n\frac{1}{k!}(n\cdots (n-k+1))^mM^kr^k\leq \sum\limits_{k=0}^n\frac{1}{k!}n^{mk}M^kr^k\leq \exp (Mrn^m).$$
\end{rem}

\begin{thm}

 If $E\subset\mathbb{C}$ then for an arbitrary
$P\in\mathbb{P}_n( \mathbb{C})$ and $m\geq 1$
$$||P'||_E\leq n^m \inf\limits_{t>0} \left(\frac{\mathcal{P}_m
(E,t)}{t}\right)||P||_E $$ and
$$||P^{(k)}||_E\leq k!\left(
\frac{n}{k}\right)^{km}\left(\inf\limits_{t>0}
\frac{\mathcal{B}_m(E,t)}{t}\right)^k||P||_E.$$

If there exist  constants $M>0,m\geq 1$ such that for all
$P\in\mathbb{P}_n( \mathbb{C})$
$$||P'||_E\leq Mn^m||P||_E,$$
then $$\mathcal{P}_m(E,t)\leq e^{Mt}.$$

If there exist  constants $M>0,m\geq 1$ such that for all
$P\in\mathbb{P}_n( \mathbb{C})$  $$||P^{(k)}||_E\leq M^k\left(
\frac{1}{k!}\right)^{m-1}n^{km}||P||_E,$$ then
$$ \mathcal{B}_m(E,t)\leq e^{mM^{1/m}t^{1/m}}.$$

\end{thm}

\begin{defin}
Define $$C_{\mathcal{P}}(E,m):=\sup\limits_{t>0}
\frac{t}{\mathcal{P}_m (E,t)}$$ and
$$C_{\mathcal{B}}(E,m ):=\sup\limits_{t>0}
\frac{t}{\mathcal{B}_m (E,t)}.$$

Since $ \mathcal{P}_m(E,t)\leq \mathcal{B}_m(E,t)$ we get
inequality
$$ C_{\mathcal{B}}(E,m)\leq C_{ \mathcal{P}}(E,m).$$

Let us note that

$$||P'||_E\leq n^m \frac{1}{C_{ \mathcal{P}}(E,m )}||P||_E $$ and
$$||P^{(k)}||_E\leq k!\left(
\frac{n}{k}\right)^{km}\left(\frac{1}{C_{ \mathcal{B}}(E,m)
}\right)^k||P||_E.$$
\end{defin}

As a corollary (to Theorem ) we get
$$C_{\mathcal{B}}(E,m)\geq e^{-m} \frac{1}{M}$$ and $$C_{\mathcal{P}}(E,m)\geq e^{-1}
\frac{1}{M}.$$
\bigskip

If $\varphi (E,r)\leq e^{Ar\alpha}$ then $ \mathcal{B}_{1/\alpha}
(E,t)\leq e^{At^\alpha}$. \medskip

If $\mathcal{B}_{m}(E,t)\leq e^{At^{1/m}}$ then $\varphi (E,r)\leq
\inf\limits_{s\geq 1}\mathcal{B}_s (E,r)\leq
\mathcal{B}_{m}(E,r)\leq e^{Ar^{1/m}}$.

We have

$$C(E)\geq \sup_{m\geq 1}C_{ \mathcal{B}}(E,m).$$
\bigskip

\bigskip
Now define
$$ \mathcal{B}_m^*(E,r)=\sup\limits_{k,n\geq 1}\varphi\left(r\left(
\frac{k}{n}\right)^m\right)^{n/k}=\sup\limits_{\sigma >0}\varphi
(r\sigma^m)^{1/\sigma}=\sup\limits_{\sigma >0}\varphi (r\sigma
)^{1/\sigma^{1/m}}=e^{A_mr^{1/m}},$$ $A_m=\sup\limits_{\sigma
>0}\frac{\log\varphi (E,\sigma )}{\sigma^{1/m}}=\sup\limits_{\sigma
>0}\frac{\rho (E,\sigma )}{\sigma^{1/m}}$.
\bigskip


$$C_{ \mathcal{B}^*}(E,m)=\sup\limits_{t>0}
\frac{t}{\mathcal{B}_m^*(E,t)}.$$

Let us observe that
$$ \mathcal{B}_m^*(E,r)=\sup\limits_{n\geq 1}
\mathcal{P}_m^*(E,rk^m)^{1/k},$$ where
$$\mathcal{P}_m^*(E,r):=\sup\limits_{k\geq 1}\varphi
(E,r/n^m)^{n}.$$ We have $\mathcal{P}_m(E,r)\leq
\mathcal{P}_m^*(E,r)$. In the case of $E=[-1,1]$ one can check
that $
\mathcal{P}_2^*([-1,1],r)=\lim\limits_{n\rightarrow\infty}\varphi
([-1,1],r/n^2)^n=e^{ \sqrt{2r}}=\mathcal{B}_2([-1,1],r)$. We shall
see that it is a consequence of a little more general facts.
\bigskip

\begin{pro} We have $\mathcal{B}_m^*(E,r)=\mathcal{B}_m(E,r)$ and
$$C_{ \mathcal{B}}(E,m)=C_{ \mathcal{B}^*}(E,m)=H_{1/m}(E),$$
where $H_\gamma (E)$ was defined in \cite{BaranCiez3}.
\end{pro}
\begin{proof} It is clear that $\mathcal{B}_m(E,r)\leq
\mathcal{B}_m^*(E,r)$. To prove opposite inequality let us observe
that by Zaharjuta-Siciak theorem (cf. \cite{SiciakAnnalesy} or
Proposition 1.3 in \cite{SiciakKukuryku}) $\varphi
(E,r)=\sup\limits_{l\geq 1}\varphi_l(E,r)^{1/l}$. Hence
$$\mathcal{B}_m^*(E,r)=\sup\limits_{k,n,l\geq 1}\varphi_l\left(E,r\left(
\frac{k}{n}\right)^m\right)^{n/kl}\leq \sup\limits_{k,n,l\geq
1}\varphi_{ln}\left(E,r\left( \frac{k}{n}\right)^m\right)^{1/kl}$$
$$=\sup\limits_{k,n,l\geq
1}\varphi_{ln}\left(E,r\left(
\frac{kl}{ln}\right)^m\right)^{1/kl}\leq \mathcal{B}_m(E,r)$$ (we
apply inequality $\varphi_l(E,r)\leq \varphi_{ln}(E,r)^{1/n}$.)
\medskip

Let us recall (cf. Definition 16 in \cite{BaranCiez3}) that for
$\gamma \in (0,1]$
$$H_\gamma (E)=1/(B(\gamma )\gamma e)^{1/\gamma},\
B(\gamma)=\sup_{r>0}\frac{\log\varphi
(E,r)}{r^\gamma}=\sup_{r>0}\frac{\rho (E,r)}{r^\gamma}.$$ We see
that $A_m=B(1/m)$.

 Now
calculate
$$ \frac{1}{C_{ \mathcal{B}^*}(E,m)}=\inf\limits_{t>0} \exp(
A_mt^{1/m}-\log t)=\exp( A_mt_0^{1/m}-\log t_0),$$ where
$t_0=(m/A_m)^{m}$. Hence
$$ C_{
\mathcal{B}^*}(E,m)=\left(eA_m\frac{1}{m}\right)^{-m}=H_{1/m}(E).$$
\end{proof}
\bigskip

\begin{pro} If $
\mathcal{P}_m^*(E,r)=\lim\limits_{n\rightarrow\infty}\varphi
(E,r/n^m)^{1/n}$ then $$ \mathcal{P}_m^*(E,rk^m)^{1/k}
=\mathcal{P}_m^*(E,r),\ k\geq 1$$ and thus
$\mathcal{B}_m(E,r)=\mathcal{P}_m^*(E,r)$.

\end{pro}
\medskip

\begin{proof}
$$
\mathcal{P}_m^*(E,rk^m)^{1/k}=\lim\limits_{n\rightarrow\infty}\varphi
(E,rk^m/(kn)^m)^{(kn)/k}=\mathcal{P}_m^*(E,r).$$
\end{proof}

\begin{rem} We know that assumption of the above proposition is
satisfied if $E=\overline{\mathbb{D}}$ ($m=1$) or $E=[-1,1]$
($m=2$). It is also true (with $m=1$) in the case of $E_R=\{
z\in\mathbb{C}:\ |h(z)|\leq R\}$. Here
$$
\mathcal{P}_1^*(E_R,r)=\mathcal{B}_1(E_R,r)=e^{r/\sqrt{g^2(R)-1}}.$$

In the general case we prove the following.
\end{rem}
\bigskip

\begin{thm}
$$ \mathcal{P}_m^*(E,rk^m)^{1/k}\leq \max(\sup\limits_{\sigma\geq 1}\varphi
(E,r\sigma )^{1/\sigma^{1/m}},\varphi
(E,r)\mathcal{P}_m^*(E,r)).$$
\end{thm}

\begin{proof}
$$
\mathcal{P}_m^*(E,rk^m)^{1/k}=\max(\max\limits_{1\leq n\leq
k}\varphi \left(E,r(k/n)^m\right)^{n/k},\max\limits_{0\leq s\leq
k-1}\sup\limits_{l\geq 1}\varphi
\left(E,rk^m/(kl+s)^m\right)^{(kl+s)/k}$$
$$\leq \max(\sup\limits_{\sigma\geq 1}\varphi (E,r\sigma
)^{1/\sigma^{1/m}},\max\limits_{0\leq s\leq k-1}\sup\limits_{l\geq
1}\varphi \left(E,rk^m/(kl+s)^m\right)^{l+s/k}
$$
$$ \leq
\max(\sup\limits_{\sigma\geq 1}\varphi (E,r\sigma
)^{1/\sigma^{1/m}},\varphi (E,r)\mathcal{P}_m^*(E,r)).$$
\end{proof}

\section{Kolmogorov-Landau type conditions.}

\subsection{Kolmogorov-Landau type theorems}

A problem related to the name Kolmogorov and Landau is the following (c.f. \cite{MitrPecFink}).

Let $M_k(p,I)=M_k(p,I,f)=||f^{(k)}||_p,\ 0\leq k\leq n$, where $f$ is a real function on the real interval $I$, $||g||_p=(\int_I|g(x)|^pdx)^{1/p},\ 1\leq p\leq \infty$. Find optimal constants $C_{nk}(p,I)$ such that for all $f\in \mathcal{C}^{n}(int(I)$
$$M_k(p,I,f)\leq C_{nk}(p,I)M_0(p,I,f)^{1-\frac{k}{n}}M_n(p,I,f)^{\frac{k}{n}},\ 0\leq k\leq n.$$
If we replace $f$ by $f'$ we get inequalities
$$M_k(p,I,f)\leq C_{n-1,k-1}(p,I)M_1(p,I,f)^{1-\frac{k-1}{n-1}}M_{n}(p,I,f)^{\frac{k-1}{n-1}},\ 0\leq k\leq n$$or equivalently 

$$\log M_k(p,I,f)\leq \log C_{n-1,k-1}(p,I)+\left(1-\frac{k-1}{n-1}\right) \log M_1(p,I,f)+\left(\frac{k-1}{n-1}\right)\log M_{n}(p,I,f),$$
$ 1\leq k\leq n.$

All results in this direction are rather hard to prove. Let us present an example (Neder inequality, c.f. (5.1) in \cite{MitrPecFink}):
$$M_k(+\infty,[a,a+L])\leq (2n)^{2n}L^{-k}M_0(+\infty,[a,a+L])+L^{n-k}M_n(+\infty,[a,a+L]).$$
Another result connected to bounded subset of $\mathbb{R}^N$ is contained in R. Redheffer and W. Walter theorem (c.f. \cite{MitrPecFink} and references given there):
\begin{itemize}\item[]\

\item[] If $G$ is a bounded domain belonging to a class $K(\theta ,H)$ (that contains a family of $N$ dimesional intervals),if we put for all $u\in C^n(G)$, $U_k=\sup\{ |D^\alpha u(x)|:\ |\alpha |=k,\ x\in G\}$, then there exists a constant $A=A(n,\theta )$ such that
$$\log U_k\leq \log A+\left(1-\frac{k}{n}\right)\log U_0+\frac{k}{n}\log U_n^*,$$ where $U_n^*=\max (U_n,h^{-n}U_0)$. 

\end{itemize}

It is rather difficult to say something about behavior of constants $A(n,\theta)$, even for special class of functions. Our goal will be give a modification of Kolmogorov-Landau type inequalities to polynomials, more precisely to factors $M_n(q,\alpha)$.
\medskip

\subsection{\bf  Kolmogorov-Landau triangle sequences.}

 \begin{defin} Consider a triangle sequence of positive numbers 
 
 $$1\leq k\leq n,\ n\in\mathbb{Z}_{+}\ \varphi (n,k)>0$$ and put $ \varphi (n,k)^{1/k}=:\psi (n,k)$. We shall say sequence $\varphi (n,k)$ {\it belongs to Kolmogorov-Landau class } $ \mathcal{KL}^*$ iff for an arbitrary $n\in\mathbb{Z}_+,\ n>1$ and every $1\leq k\leq n$

$$\log\psi(n,k)\leq \left(1-\frac{\log k}{\log n}\right)\log\psi (n,1)+\frac{\log k}{\log n}\log\psi (n,n).$$ In such a situation we shall write $\varphi (n,k)\in \mathcal{KL}^*$.
Similarly, we shall say sequence $\varphi (n,k)$ {\it belongs to Kolmogorov-Landau class } $ \mathcal{KL}$ iff there exitsts a positive constant $C$ such that for an arbitrary $n>1$ and every $1\leq k\leq n$

$$ \log\psi(n,k)\leq \log C+ \left(1-\frac{\log k}{\log n}\right)\log\psi (n,1)+\frac{\log k}{\log n}\log\psi (n,n).$$
We shall write $\varphi (n,k)\in \mathcal{KL}$.

Obviously $\varphi (n,k)\in \mathcal{KL}^*\Rightarrow\ \varphi (n,k)\in \mathcal{KL}^*$. We also see that $\mathcal{KL}^*$ is a kind of convexity property and thus $\mathcal{KL}$ is a kind of weak convexity condition.
\medskip

\end{defin}
\medskip

\begin{exa} It is easy to check that the following sequences belong to $\mathcal{KL}^*$.
\medskip 
\begin{itemize}

\item[(1)] $\varphi (n,k)=k^k$

\item[(2)] $\varphi (n,k)=e^{\sigma k}$

\item[(3)] $\varphi (n,k)= \left(\frac{n}{k}\right)^{km}$

\item[(4)] $\varphi (n,k)=k^k\cdot \left(\frac{n}{k}\right)^{km}$

\item[(5)] $\varphi (n,k)=e^{k\sigma}\cdot k^k\cdot \left(\frac{n}{k}\right)^{km}$

\item[(6)] $\varphi (n,k)=2^{n-k}n^k.$
\end{itemize}

In the  examples below we used the following simple observations.

\begin{pro}

 If $\varphi_1(n,k),\varphi_2(k,n)\in \mathcal{KL},\ m>0$ then 
 
\begin{itemize}
\item[(a)] 
 $\varphi_1(n,k)\varphi_2(n,k)\in \mathcal{KL}$,
 \medskip
 
 \item[(b)] $\varphi_1(n,k)^m\in \mathcal{KL}$,
\medskip

\item[(c)] $\max (\varphi_1(n,k),\varphi_2(n,k))\in \mathcal{KL}$.
\medskip

\item[(d)] If there exist positive constant $A_1,A_2$ such that $A_1^k\leq \frac{\varphi_1(n,k)}{\varphi_2(n,k)}\leq A_2^k$ then 
$$\varphi_1(k,n)\in \mathcal{KL}\ \Leftrightarrow \ \varphi_2(k,n)\in \mathcal{KL}.$$

\end{itemize}
\end{pro}

\end{exa}

\begin{exa} In the following cases we can check that a sequence belongs to $\mathcal{KL}$ with a given constant $C$ (usually not optimal). We refer to Mitrinovi\'c book or to Wikipedia for needed inequalities for factorials $n!$  and Newton symbols $\binom{n}{k}$ and left calculations to the reader.

\begin{itemize}

\item[(1)] $\varphi (n,k)=k!,\  \log C=\frac{11}{12}+\frac{1}{2e}+\frac{1}{2}\log (2\pi ).$

\item[(2)] $\varphi (n,k)=\binom{n}{k}^m,\ C=e^m. $

\item[(3)] $\varphi (n,k)=k!\cdot \binom{n}{k}^m,\ \log C=m+\frac{11}{12}+\frac{1}{2e}+\frac{1}{2}\log (2\pi ). $

\item[(4)]  $\varphi (n,k)=k!e^{-k}\exp ((1+1/s)k^{\frac{s}{1+s}}n^{\frac{1}{1+s}}),\ 
 \log C=\frac{11}{12}+\frac{1}{2e}+\frac{1}{2}\log (2\pi )+1+1/s$,
where $0<s\leq 1$.

\item[(5)] $\varphi (n,k)=\left(\frac{n}{k}\right)^{km}(1+\log (n/k))^{mk}.$

\end{itemize}
\end{exa}
\bigskip

\subsection{\bf  Kolmogorov-Landau norms.}
\bigskip

Let $q(P)=||P||$ be a norm in $\mathbb{P}(\mathbb{C}^N)$. 

Put $$M(n,k)=M_q(n,k):=\sup\{ ||D^{\alpha}P||:\ |\alpha|=k,\ k\leq  \deg P\leq n,\ ||P||=1\} .$$

\begin{defin} \ 
\begin{itemize}

\item[(1)] $q\in \mathcal{KL}$ if $M_q(n,k)\in \mathcal{KL}$.
\medskip

\item[(2)] $q\in \mathcal{KL}_*$ if there exists $\varphi (n,k)\in \mathcal{KL}$ such that $M_q(n,k)\leq \varphi (n,k)$ ($\varphi (n,k)$ will be called $\mathcal{KL}$ {\it  majorant}). 
\end{itemize}

\end{defin}
\bigskip

\begin{exa}\ 
\begin{itemize}

\item[(1)] Let $E=\bar{\mathbb{D}}$, $q(P)=||P||_E$. There is well known that $$M_q(n,k)=n(n-1)\cdots (n-k+1)=k!\cdot \binom{n}{k}.$$ By the Example 1.3 (3) we get $q\in \mathcal{KL}$.
\medskip

\item[(2)] Consider $E=[-1,1]$, $q(P)=||P||_E$. The famous Vladimir Markov inequality gives 
$$M_q(n,k)=\frac{(n(n-1)\cdot (n-k+1))^2}{1\cdot 3\cdot (2k-1)}\leq k!\binom{n}{k}^2.$$ Hence, by the Example 1.3 (3) we get $q\in \mathcal{KL}_*$. One can also check that $\varphi (n,k)=\frac{(k!)^3}{(2k)!}\in \mathcal{KL}$, which gives $q\in \mathcal{KL}$. Applying recent result by G. Sroka \cite{Sroka} one can check that $q_p\in \mathcal{KL}_*$, where $q_p(P)=\left(\frac{1}{2}\int_E|P(x)|^pdx\right)^{1/p},\ p\geq 1$.
\medskip

\item[(3)] Let $q(P)=||P||=\sum\limits_{j=0}^\infty \left(\frac{1}{j!}\right)^m|P^{(j)}(0)|\tau^j,\ \tau >0,\ m\geq 0$. 

If $P(x)=a_0+a_1x+\dots +a_nx^n$ then $||P||=\sum\limits_{j=0}^\infty \left(\frac{1}{j!}\right)^{m-1}|a_j|\tau^j .$

Put $S_{q,n}:=\{ P\in\mathbb{P}_n(\mathbb{C}):\ ||P||\leq 1\}$ - this is a convex symmetric body in finite dimensional vector space $\mathbb{P}_n(\mathbb{C})$.
Then one can calculate that

$$\text{\rm extr}(S_{q,n})=\{ P_j(x)=\zeta_j(j!)^{m-1}\tau^{-j}x^j:\ |\zeta_j|=1,\ j=0,\dots ,n\}.$$
Hence $\sup\{ ||P^{(k)}||:\ P\in S_{q,n}\}=\max\{ ||P_j^{(k)}||:\ j=1,\dots, n\}$. 

Since $||x^l||=(1/l!)^{m-1}\tau^l$, we get 
$$||P_j^{(k)}||=\tau^{-k}(j(j-1)\cdots (j-k+1))^m.$$ Hence
$$M_q(n,k)=\max\limits_{k\leq j\leq n}||P_j^{(k)}||=\tau^{-k}(n(n-1)\cdots (n-k+1))^m.$$ Consequently $q\in \mathcal{KL}$.

\medskip

\item[(4)] If $E\subset\mathbb{C}^N$ is a Bernstein set ($||D^\alpha P||_E\leq B^{|\alpha |}(\deg P)^{|\alpha |}||P||_E$) then $q=||\cdot ||_E\in \mathcal{KL} $.
\medskip

\item[(5)] If $E$ is a compact subset of $\mathbb{C}$ then $q(P)=\sum\limits_{j=1}^\infty \frac{1}{j!}||P^{(j)}||_E\tau^j\in \mathcal{KL}$.

\item[(6)] If $q\in \mathcal{AM}$ ($M_q(n,k)\leq B^{k}n^{km}$) or $q\in\mathcal{VM}$ ($M_q(n,k)\leq B^kn^{km}/(k!)^{m-1}$) then $q\in \mathcal{KL}_*$.

\end{itemize}

\end{exa}
\bigskip

Let us formulate the main results of this paper.
\bigskip

\begin{thm}

If $q\in \mathcal{KL}$ then two conditions are equivalent

\begin{itemize}

\item[(1)] $q\in \mathcal{VM}$;

\item[(2)] $q\in \mathcal{AM},\ M(n,n)\leq A^n n!$.
\end{itemize}
\end{thm}
\bigskip

\begin{thm} Two conditions are equivalent 
\begin{itemize}
\item[(1)] $q\in \mathcal{VM}$;

\item[(2)] there exists $\varphi (n,k)$  a $\mathcal{KL}$ majorant such that $\varphi (n,1)\leq An^\alpha,\ \varphi (n,n)\leq B^nn!$.
\end{itemize}
\end{thm}
\bigskip

\begin{rem} If $q$ is a norm in $\mathbb{C}^N$ then we conjecture $q\in \mathcal{KL}$.
\end{rem}
\bigskip

If the above conjecture is true, then applying \cite{Ada} we get the following.

\begin{cor} If $E\subset\mathbb{C}$ then $E\in\mathcal{AM}
\Leftrightarrow\ E\in \mathcal{VM}$.
\end{cor}

\begin{rem} Let $q$ be the norm in Example 6.6 (3) with $m>1$. Then $q\in \mathcal{KL}$, $q\in \mathcal{AM}$ but $q\not\in\mathcal{VM}$ as it was proved in \cite{MirekAgnieszkaBeataPawel}.\end{rem}

\begin{rem} Let us consider the following condition: $q\in \mathcal{M}_*(a,m)$ if and only if
$$\left(\frac{M_n(q,l)}{l!}\right)^{1/l}/\left(\frac{M_n(q,k)}{k!}\right)^{1/k}\leq a\left(\frac{k}{l}\right)^m,\ \forall n\geq 1,\ 1\leq l\leq k\leq n,$$ where $m\geq 1$ is a constant. In particular, if $k=n$, we obtain a condition
$$M_n(q,l)\leq a^ll!\left(\frac{n}{l}\right)^{lm}\left(M_n(q,n)^{1/n}/n!^{1/n}\right)^l.$$
Hence, if $M_n(q,n)\leq b^nn!$, we obtain V. Markov's inequality
$$M_n(q,l)\leq (a(q)b(q))^l n^{ml}/l!^{m-1}. $$ Let us note that considered condition $\mathcal{M}_*(a,m)$ is not satisfied if $q(P)=||P||_{\mathbb{D}\cup\{z_0\}}$,\ with $|z_0|>1$. On the other hand this condition is satisfied if $q(P)$ is a norm from Example  6.6(3).
\medskip

Now we can formulate the following question: thus
$$q\in\mathcal{AM}(M,m) \Rightarrow\  q\in \mathcal{M}_*(a',m')$$ or (a weaker condition)
$$q\in\mathcal{AM}(M,m) \Rightarrow\ M_n(q,l)\leq a'^ll!\left(\frac{n}{l}\right)^{lm'}\left(M_n(q,n)^{1/n}/n!^{1/n}\right)^l?$$
\end{rem}
\bigskip

\noindent \textbf{Acknowledgement.}
The work was partially supported by the National Science Centre (NCN), Poland No. 2013/11/B/ST1/03693.

\end{document}